\documentclass[reqno,11pt]{article}
\usepackage{a4wide,color,eucal,enumerate,mathrsfs}
\usepackage[normalem]{ulem}
\usepackage{amsmath,amssymb,epsfig,bbm}
\numberwithin{equation}{section}

\usepackage[latin1]{inputenc}

\usepackage[pdfborder={0 0 0}]{hyperref}

\newenvironment{proof}{\removelastskip\par\medskip   
\noindent{\em proof} \rm}{\penalty-20\null\hfill$\square$\par\medbreak}

\newtheorem{theorem}{Theorem}[section]

\newtheorem{corollary}[theorem]{Corollary}
\newtheorem{lemma}[theorem]{Lemma}
\newtheorem{proposition}[theorem]{Proposition}

\newtheorem{definition}[theorem]{Definition}

\newcommand{\CD}{{\sf CD}}
\newcommand{\RCD}{{\sf RCD}}
\newcommand{\MCP}{{\sf MCP}}

\newcommand{\ppi}{{\mbox{\boldmath$\pi$}}}

\newcommand{\supp}{\mathop{\rm supp}\nolimits}
\renewcommand{\d}{{\mathrm d}}
\newcommand{\N}{\mathbb{N}}
\newcommand{\R}{\mathbb{R}}

\newcommand{\mm}{\mathfrak m}

\newcommand{\sfd}{{\sf d}}
\newcommand{\prob}[1]{\mathscr P(#1)}
\newcommand{\probt}[1]{\mathscr P_2(#1)}

\newcommand{\geo}{{\rm{Geo}}}
\newcommand{\e}{{\rm{e}}}
\newcommand{\gopt}{{\rm{OptGeo}}}

\newcommand{\lims}{\varlimsup}

\newcommand{\restr}[1]{\lower3pt\hbox{$|_{#1}$}}

\title{Optimal maps and exponentiation on finite dimensional spaces with Ricci curvature bounded from below}
\begin{document}

\author{Nicola Gigli\
   \thanks{University of Nice, \textsf{nicola.gigli@unice.fr}}
   \and
   Tapio Rajala
   \thanks{University of Jyv\"askyl\"a, \textsf{tapio.m.rajala@jyu.fi}}
   \and
   Karl-Theodor Sturm
   \thanks{Universit\"at Bonn, \textsf{sturm@uni-bonn.de}}
   }
\maketitle

\begin{abstract}
We prove existence and uniqueness of optimal maps on $\RCD^*(K,N)$ spaces under the assumption that the starting measure is absolutely continuous. We also discuss  how this result naturally leads to the notion of exponentiation. 
\end{abstract}

\tableofcontents

\section{Introduction}
A basic problem in optimal transportation is the question on whether optimal plans are unique and induced by maps. The crucial result in this direction is the celebrated one of Brenier \cite{Brenier87} granting that for $\mu,\nu\in\probt{\R^d}$ with $\mu$ absolutely continuous w.r.t. the Lebesgue measure and cost=squared-distance, indeed optimal plans are unique and induced by maps. An important generalization has been given by McCann \cite{McCann01} for the same problem on Riemannian manifolds: he shows that the unique optimal map can be written as  $\exp(-\nabla\varphi)$, where $\varphi$ is a Kantorovich potential. As a byproduct of McCann's argument, we also know that for $\mu$-a.e. $x$ the geodesic connecting $x$ to $\exp(-\nabla\varphi(x))$ is unique. We can express both the fact that the unique  optimal plan is induced by a map and the uniqueness of geodesics by looking at the optimal transport problem as a dynamical problem, i.e. by minimizing
\[
\iint_0^1|\dot\gamma_t|^2\,\d t\,\ppi(\gamma), 
\] 
over all measures $\ppi$ on $C([0,1],M)$ such that $(\e_0)_\sharp\ppi=\mu$, $(\e_1)_\sharp\ppi=\nu$, $\e_t:C([0,1],M)\to M$ being the evaluation map given by $\e(\gamma):=\gamma_t$. Then McCann's result can be read as the uniqueness of the minimizer $\ppi$ and the fact that such $\ppi$ is induced by some map $T:M\to C([0,1],M)$ (which in fact takes its values in the space of constant speed minimizing geodesics) in the sense that $\ppi=T_\sharp\mu$. We refer to \cite{Villani09} and \cite{AmbrosioGigli11} for an overview of the subject.

\medskip

In the pioneering works of Lott-Villani \cite{Lott-Villani09} and Sturm \cite{Sturm06I}, \cite{Sturm06II}, an abstract notion of lower Ricci curvature bound on metric measure spaces has been given, and since then a great interest has been given by the community to the understanding of the geometric/analytic properties of these spaces. In \cite{AmbrosioGigliSavare11-2}, a strengthening of the original $\CD(K,\infty)$ condition as defined by Lott-Sturm-Villani has been proposed: this new condition, called Riemannian Ricci curvature bound and denoted by $\RCD(K,\infty)$, enforces in some weak sense a Riemannian-like behavior of the space. According to the analysis done in \cite{Erbar-Kuwada-Sturm13}, \cite{AMS}, \cite{Gigli13} a natural finite-dimensional analogue of the $\RCD(K,\infty)$ notion can be given by requiring that the space is both $\RCD(K,\infty)$ and satisfies the (reduced) curvature-dimension condition $\CD^*(K,N)$ as defined in \cite{BacherSturm10}.

Aim of this short note is to prove the analogue of  Brenier-McCann's theorem  on $\RCD^*(K,N)$ spaces, our result being:
\begin{theorem}[Optimal maps]\label{thm:1}
Let $K\in\R$, $N\in[1,\infty)$ and $(X,\sfd,\mm)$  an $\RCD^*(K,N)$ space. Then for every $\mu,\nu\in\probt X$ with $\mu\ll\mm$ there exists a unique plan $\ppi\in\gopt(\mu,\nu)$. Furthermore, this plan is induced by a map and concentrated on a set of non-branching geodesics, i.e. there is a Borel set $\Gamma\subset C([0,1],X)$ such that $\ppi(\Gamma)=1$ and for every $t\in[0,1)$ the map $\e_t:\Gamma\to X$ is injective.
\end{theorem}
Here by $\gopt(\mu,\nu)$ we are denoting the set of minimizers of the dynamical version of the optimal transport as discussed above.
To some extent, the `hard work' needed for the proof of this result has been already carried out in \cite{RajalaSturm12} and \cite{Gigli12a} where it has been proved the following theorem:
\begin{theorem}[Optimal maps on $\RCD(K,\infty)$ spaces]\label{thm:2a}
Let $K\in\R$ and $(X,\sfd,\mm)$  an $\RCD(K,\infty)$ space. Then for every $\mu,\nu\in\probt X$ with $\mu,\nu\ll\mm$ there exists a unique plan $\ppi\in\gopt(\mu,\nu)$ and  this plan is induced by a map and concentrated on a set of non-branching geodesics.
\end{theorem}
More precisely,  in \cite{RajalaSturm12} it has been worked around the delicate issue concerning the non-branching assumption, showing that on $\RCD(K,\infty)$ spaces every optimal geodesic plan between absolutely continuous measures must be concentrated on a set of non-branching geodesics. Then, still in \cite{RajalaSturm12}, it has been observed how such  result coupled with the technique used in \cite{Gigli12a} to prove existence and uniqueness of optimal maps in the non-branching case yield Theorem \ref{thm:2a}.

Here we  start from this results and obtain Theorem \ref{thm:1} using the enhanced compactness granted  by the finite dimensionality together with quite standard ideas in optimal transport theory.

An  interesting fact about Theorem \ref{thm:1} is that it can be equivalently reformulated in the following way:
\begin{theorem}[Exponentiation]\label{thm:2b}
Let $K\in\R$, $N\in[1,\infty)$, $(X,\sfd,\mm)$  an $\RCD^*(K,N)$ space and $\varphi:X\to\R$ a $c$-concave function ($c=\frac{\sfd^2}{2}$). Then  for $\mm$-a.e. $x\in X$ there exists exactly one geodesic $\gamma$ such that $\gamma_0=x$ and $\gamma_1\in\partial^c\varphi(x)$.
\end{theorem}
This result can be naturally interpreted as a definition of what is the exponential map evaluated at  `minus the gradient of a $c$-concave function $\varphi$': for every $x\in X$ such that the geodesic $\gamma$  with $\gamma_0=x$ and $\gamma_1\in\partial^c\varphi(x)$ is unique, we put $\exp(-t\nabla\varphi):=\gamma_t$, thus somehow `reversing' the proof of Brenier-McCann theorem. The role of Theorem \ref{thm:2b} is to ensure that this map  is well defined for $\mm$-a.e. $x\in X$.

Notice that to some extent Theorem \ref{thm:2b} is the best one we can expect about exponentiation on a metric measure space. To see why just consider the case of a smooth complete Riemannian manifold $M$ with boundary. Then given $x\in M$ and $v\in T_xM$, the value of $\exp(v)$ is well defined only if there is $y\in M$ such that $\nabla\frac{\sfd^2(\cdot,y)}2=-v$ (neglecting smoothness issues), and functions of the kind $\frac{\sfd^2(\cdot,y)}2$ are the prototype of $c$-concave functions. 

Theorem \ref{thm:1} has some simple but interesting consequences, the first being:
\begin{corollary}\label{cor:1}
Let $K\in\R$, $N\in[1,\infty)$ and $(X,\sfd,\mm)$  an $\RCD^*(K,N)$ space. Then for every $x\in \supp(\mm)$ the following holds: for $\mm$-a.e. $y$ there is only one geodesic connecting $y$ to $x$.
\end{corollary} 
This can be easily seen choosing $\nu:=\delta_x$ in Theorem \ref{thm:1}.
In \cite{R2012a} the conclusion of Corollary \ref{cor:1} was proven under the assumption that the $\CD(K,N)$ condition holds along every geodesic.
However, $\RCD^*(K,N)$ a priori only gives the $\CD^*(K,N)$ condition along every geodesic between any two measures with bounded densities, see \cite{Erbar-Kuwada-Sturm13}.
Thus Corollary \ref{cor:1} is not a direct consequence of \cite[Theorem 4]{R2012a}.
A further consequence of this corollary is the following:
\begin{corollary}\label{cor:2}
Let $K\in\R$, $N\in[1,\infty)$ and $(X,\sfd,\mm)$  an $\RCD^*(K,N)$ space. Then the space satisfies the $\MCP(K,N)$ condition.
\end{corollary} 
From \cite{R2012b} we know that every $\CD(K,N)$ space satisfies the $\MCP(K,N)$ condition in the sense of \cite{O2007},
meaning that between any absolutely continuous measure and a dirac mass there exists a geodesic that satisfies the $\MCP(K,N)$ condition.
In Corollary \ref{cor:2} we obtain a more strict version of the $\MCP(K,N)$ condition, considered in \cite{Sturm06II}, with
a global selection of distributions of geodesics between points such that using these geodesics
the $\MCP(K,N)$ condition always holds. Since by Corollary \ref{cor:1} the geodesics are essentially unique, in fact
any choice of geodesics in an $\RCD^*(K,N)$ space will work for the $\MCP(K,N)$ condition.

The difficult part in proving Corollary \ref{cor:2} relies in proving a sort of self-improving property for the $\CD^*(K,N)$ condition: this has been the scope of \cite{Cavalletti-Sturm12}, where such result has been proved under the non-branching assumption. Yet, such additional hypothesis was made only to get the result of Corollary \ref{cor:1} above. Given that in the $\RCD^*(K,N)$ it holds without the a priori non-branching assumption, Corollary \ref{cor:2} follows.

A final remark which is worth to make, in particular in connection with Sobolev calculus as developed in \cite{AmbrosioGigliSavare11}, is the following:
\begin{corollary}
Let $K\in\R$, $N\in[1,\infty)$, $(X,\sfd,\mm)$  an $\RCD^*(K,N)$ space,  $\mu,\nu\in\probt X$ with $\mu\ll\mm$ and $\ppi\in\gopt(\mu,\nu)$ be the  unique optimal geodesic plan given by Theorem \ref{thm:1}. Then $(\e_t)_\sharp\ppi\ll\mm$ for every $t\in[0,1)$. 

Furthermore, if $\mu,\nu$ have bounded support (resp. $K=0$) and $\mu$ and has density bounded above by some constant $C$, then $(\e_t)_\sharp\ppi\leq C(t)\mm$ for any $t\in[0,1)$ and some constant $C(t)$ depending only on $C,t,K,N$ and the supports of $\mu,\nu$ (resp. on $C,t,N$). If $K<0$ and either $\mu$ or $\nu$ have unbounded support, then the optimal geodesic plan $\ppi\in\gopt(\mu,\nu)$ can be written as $\ppi=\sum_{n\in\N}\ppi_n$ with $\ppi_n$ non negative Borel measures on $\geo(X)$ such that  $(\e_t)_\sharp\ppi\leq C_n(t)\mm$ for any $t\in[0,1)$ and some constants $C_n(t)$ depending only on $C,t,K,N,n$.
\end{corollary}
The simple proof follows by localizing the $\CD^*(K,N)$ condition along the optimal geodesic plan.

\smallskip
\noindent {\bf Acknowledgement.}
This paper was partly written during the program ``Interactions Between Analysis and Geometry'' at the Institute for Pure and Applied Mathematics (IPAM)
at University of California, Los Angeles. The authors thank the institute for the excellent research environment.
T.R. also acknowledges the support of the Academy of Finland project no. 137528.

\section{Preliminaries}

We assume the reader to be familiar with optimal transport and the definition of spaces with Ricci curvature bounded from below in the sense of Lott-Sturm-Villani. Here we just recall some basic notation.

Given a geodesic, complete and separable metric space $(X,\sfd)$, the set $\probt X$ is the set of Borel probability measures on it with finite second moment. By $\geo(X)$ we denote the space of constant speed minimizing geodesics on $X$ endowed with the $\sup$-distance.

Given such metric space $(X,\sfd)$ and $\mu,\nu\in\probt X$, a Borel probability measure $\ppi$ on $\geo(X)$ is called optimal geodesic plan from $\mu$ to $\nu$ provided $(\e_0)_\sharp\ppi=\mu$, $(\e_1)_\sharp\ppi=\nu$  and it achieves the minimum of
\[
\int \sfd^2(\gamma_0,\gamma_1)\,\d\ppi(\gamma),
\]
among all Borel probability measure $\ppi'$ on $\geo(X)$ such that $(\e_0)_\sharp\ppi'=\mu$, $(\e_1)_\sharp\ppi'=\nu$.  The set of all optimal geodesic plans is denoted by $\gopt(\mu,\nu)$. Notice that $\gopt(\mu,\nu)$ is never empty under the above assumption.

A function $\varphi:X\to\R\cup\{-\infty\}$ not identically $-\infty$ is called $c$-concave provided there is $\psi:X\to\R\cup\{-\infty\}$ such that
\[
\varphi(x)=\inf_{y\in X}\frac{\sfd^2(x,y)}{2}-\psi(y).
\]
Given a $c$-concave function $\varphi$, its $c$-transform $\varphi^c:X\to\R\cup\{+\infty\}$ is defined by
\[
\varphi^c(y):=\inf_{x\in X}\frac{\sfd^2(x,y)}{2}-\varphi(x).
\] 
It turns out that for every $c$-concave function $\varphi$ it holds $\varphi^{cc}=\varphi$. The $c$-superdifferential $\partial^c\varphi$ of a $c$-concave function $\varphi$ is the subset of $X^2$ of those couples $(x,y)$ such that
\[
\varphi(x)+\varphi^c(y)=\frac{\sfd^2(x,y)}2,
\]
and for $x\in X$, the set $\partial^c\varphi(x)\subset X$ is the set of those $y$'s such that $(x,y)\in\partial^c\varphi(y)$.

It can be proved that a Borel probability measure $\ppi$ on $\geo(X)$ belongs to $\gopt((\e_0)_\sharp\ppi,(\e_1)_\sharp\ppi)$ if and only if there is a $c$-concave function $\varphi$ such that $\supp(\e_0,\e_1)_\sharp\ppi\subset\partial^c\varphi$. Any such $\varphi$ is called Kantorovich potential from $(\e_0)_\sharp\ppi$ to $(\e_1)_\sharp\ppi$. It is then easy to check that for any Kantorovich potential $\varphi$ from $\mu$ to $\nu$, every  $\ppi\in\gopt(\mu,\nu)$ and every $t\in[0,1]$,  the function $t\varphi$ is a Kantorovich potential from $\mu$ to $(\e_t)_\sharp\ppi$.

Notice that Kantorovich potentials can be chosen to satisfy the following property, slightly stronger than $c$-concavity:
\[
\varphi(x)=\inf_{y\in\supp(\nu)}\frac{\sfd^2(x,y)}{2}-\varphi^c(y),
\]
which shows in particular that if $\supp(\nu)$ is bounded, then $\varphi$ can be chosen to be locally Lipschitz.

We turn to the formulation of the $\CD^*(K,N)$ condition, coming from  \cite{BacherSturm10}, to which we also refer for a detailed discussion of its relation with the $\CD(K,N)$ condition
 (see also \cite{Cavalletti-Sturm12} and \cite{Cavalletti12}).

Given $K \in \R$ and $N \in [1, \infty)$, we define the distortion coefficient $[0,1]\times\R^+\ni (t,\theta)\mapsto \sigma^{(t)}_{K,N}(\theta)$ as
\[
\sigma^{(t)}_{K,N}(\theta):=\left\{
\begin{array}{ll}
+\infty,&\qquad\textrm{ if }K\theta^2\geq N\pi^2,\\
\frac{\sin(t\theta\sqrt{K/N})}{\sin(\theta\sqrt{K/N})}&\qquad\textrm{ if }0<K\theta^2 <N\pi^2,\\
t&\qquad\textrm{ if }K\theta^2=0,\\
\frac{\sinh(t\theta\sqrt{K/N})}{\sinh(\theta\sqrt{K/N})}&\qquad\textrm{ if }K\theta^2 <0.
\end{array}
\right.
\]
\begin{definition}[Curvature dimension bounds]
Let $K \in \R$ and $ N\in[1,  \infty)$. We say that a m.m.s.  $(X,\sfd,\mm)$
 is a $\CD^*(K,N)$-space if for any two measures $\mu_0, \mu_1 \in \prob X$ with support  bounded and contained in $\supp(\mm)$ there
exists a measure $\ppi \in \gopt(\mu_0,\mu_1)$ such that for every $t \in [0,1]$
and $N' \geq  N$ we have
\begin{equation}\label{eq:CD-def}
-\int\rho_t^{1-\frac1{N'}}\,\d\mm\leq - \int \sigma^{(1-t)}_{K,N'}(\sfd(\gamma_0,\gamma_1))\rho_0^{-\frac1{N'}}+\sigma^{(t)}_{K,N'}(\sfd(\gamma_0,\gamma_1))\rho_1^{-\frac1{N'}}\,\d\ppi(\gamma)
\end{equation}
where for any $t\in[0,1]$ we  have written $(\e_t)_\sharp\ppi=\rho_t\mm+\mu_t^s$  with $\mu_t^s \perp \mm$.
\end{definition}
Notice that if $(X,\sfd,\mm)$ is a $\CD^*(K,N)$-space, then so is $(\supp(\mm),\sfd,\mm)$, hence it is not restrictive to assume that $\supp(\mm)=X$, a hypothesis that we shall always implicitly do from now on. Also, for any $\CD^*(K,N)$ space $(X,\sfd,\mm)$ we have that $(X,\sfd)$ is geodesic and proper.

In \cite{AmbrosioGigliSavare11-2} (see also \cite{AmbrosioGigliMondinoRajala12}) an enforcement of the curvature condition $\CD(K,\infty)$ as defined by Lott-Villani and Sturm in \cite{Lott-Villani09} and \cite{Sturm06I} has been proposed. This condition, called Riemannian Ricci curvature bound and denoted by $\RCD(K,\infty)$, enforces in some weak sense a Riemannian-like structure of the space. For our purposes, it is not necessary to recall the quite technical definition, but only the following crucial result, proved in  \cite{RajalaSturm12} (see also \cite{Gigli12a}):
\begin{theorem}[Optimal maps in $\RCD(K,\infty)$ spaces]\label{thm:optmap}
Let $(X,\sfd,\mm)$ be an  $\RCD(K,\infty)$ space and $\mu,\nu\in\probt X$ two measures absolutely continuous w.r.t. $\mm$. 

Then there exists a unique $\ppi\in\gopt(\mu,\nu)$ and this plan is induced by a map and concentrated on a set of non-branching geodesics, i.e. for any $t\in[0,1]$ there exists a Borel map $T_t:X\to\geo(X)$ such that $\ppi=(T_t)_\sharp(\e_t)_\sharp\ppi$. 
\end{theorem}

Finally we recall the definition of $\RCD^*(K,N)$ spaces as given in \cite{Erbar-Kuwada-Sturm13} (see also \cite{AMS}):
\begin{definition}[$\RCD^*(K,N)$ spaces]
Let $K\in \R$ and $N\in[1,\infty)$. We say that $(X,\sfd,\mm)$ is an $\RCD^*(K,N)$ space provided it is both $\CD^*(K,N)$ and $\RCD(K,\infty)$.
\end{definition}

\section{Exponentiation and optimal maps}
We start with the following simple result which shows how the use of Theorem \ref{thm:optmap} allows for the localization of the $\CD^*(K,N)$ condition along a geodesic connecting two absolutely continuous measures.
\begin{proposition}\label{prop:varphi}
Let $(X,\sfd,\mm)$ be an $\RCD^*(K,N)$ space and  $\mu_i=\rho_i\mm\in\probt X$, $i=0,1$,  two given measures. Let $\ppi\in\gopt(\mu_0,\mu_1)$ be the unique optimal geodesic plan from $\mu_0$ to $\mu_1$ given by Theorem \ref{thm:optmap} and put $\mu_t:=(\e_t)_\sharp\ppi$. Then $\mu_t\ll\mm$ for every $t\in[0,1]$ and writing $\mu_t=\rho_t\mm$ for every $0\leq t\leq r\leq s\le 1$ we have
\begin{equation}
\label{eq:boundpoint}
\rho_r(\gamma_r)^{-\frac1N}\geq \rho_t(\gamma_t)^{-\frac1N}\sigma^{(\frac{s-r}{s-t})}_{K,N}(\sfd(\gamma_t,\gamma_s))+\rho_s(\gamma_s)^{-\frac1N}\sigma^{(\frac{r-t}{s-t})}_{K,N}(\sfd(\gamma_t,\gamma_s)),\qquad\ppi-a.e.\ \gamma.
\end{equation}
\end{proposition}

\begin{proof} We start by proving that $\mu_t\ll\mm$ for every $t\in[0,1]$.  Fix $\bar x\in X$ and for $M>0$ let $G_M\subset \geo(X)$ be defined by 
\[
G_M:=\Big\{\gamma\in \geo(X)\ :\ \rho_0(\gamma_0),\rho_1(\gamma_1),\sfd(\gamma_0,\bar x),\sfd(\gamma_1,\bar x)\leq M\Big\}.
\]
For $M$ large enough we have $\ppi(G_M)>0$, thus the plan $\ppi_M:=c_M\ppi\restr{G_M}$ is well defined,  $c_M:=\ppi(G_M)^{-1}$ being the normalizing constant. Put $\mu_0^M:=(\e_0)_\sharp\ppi_M$, $\mu_1^M:=(\e_1)_\sharp\ppi_M$ and notice that $\mu_0^M,\mu_1^M\ll\mm$ and that by construction and since  optimality is stable by restriction we get  $\ppi_M\in\gopt( \mu_0^M, \mu_1^M)$. Hence the uniqueness part of Theorem \ref{thm:optmap} yields that  $\ppi_M$ is the only optimal plan from $\mu_0^M$ to $\mu_1^M$. Being $(X,\sfd,\mm)$ a $\CD(K,N)$ space it is also a  $\CD(K,\infty)$ space and thus fact that ${\rm Ent}_\mm(\mu_0^M),{\rm Ent}_\mm(\mu_1^M)<\infty$ (because both have bounded densities) give ${\rm Ent}_\mm((\e_t)_\sharp\ppi_M)<\infty$ for every $t\in[0,1]$. In particular, $(\e_t)_\sharp\ppi_M\ll\mm$ for every $t\in[0,1]$. Since $(\e_t)_\sharp\ppi_M\uparrow (\e_t)_\sharp\ppi=\mu_t$ as $M\to\infty$, we deduce $\mu_t\ll\mm$ for every $t\in[0,1]$.

We turn to \eqref{eq:boundpoint}. Assume for a moment $t=0$, $s=1$ and that the supports of $\mu_0,\mu_1$ are bounded and notice that in this case to prove \eqref{eq:boundpoint} is equivalent to prove that for any Borel set $G\subset\geo(X)$ it holds
\begin{equation}
\label{eq:intcd}
\begin{split}
-\int_G\rho_{r}^{-\frac1N}(\gamma_{r})\,\d\ppi(\gamma)\leq& -\int_G\rho_0(\gamma_0)^{-\frac1N}\sigma^{(1-t)}_{K,N}(\sfd(\gamma_0,\gamma_1))+\rho_1(\gamma_1)^{-\frac1N}\sigma^{(t)}_{K,N}(\sfd(\gamma_0,\gamma_1))\,\d\ppi(\gamma).
\end{split}
\end{equation}
Fix such Borel set $G\subset\geo(X)$, assume without loss of generality that $\ppi(G)>0$ and define $\ppi_G:=\ppi(G)^{-1}\ppi\restr G$.  Let $T_t:X\to\geo(X)$ be the maps given by Theorem \ref{thm:optmap} and notice that the identity $\ppi=(T_t)_\sharp(\e_t)_\sharp\ppi$ ensures $(\e_t)_\sharp\ppi_G=\ppi(G)^{-1}\chi_G\circ T_t(\e_t)_\sharp\ppi$. In other words, letting $\rho_{G,t}\mm=(\e_t)_\sharp\ppi_G$, a direct consequence of the fact that $\ppi$ is concentrated on a set of non-branching geodesics is that we have
\begin{equation}
\label{eq:simpleresc}
\rho_{G,t}(\gamma_t)=\ppi(G)^{-1}\rho_t(\gamma_t),\qquad\ppi-a.e.\ \gamma\in G.
\end{equation}
It is clear that $\ppi_G$ is optimal from $\rho_{G,0}\mm$ to $\rho_{G,1}\mm$ and by the uniqueness part of Theorem \ref{thm:optmap} we know that it is the only optimal plan, hence the $\CD^*(K,N)$ condition and the fact that $\rho_{G,0}\mm,\rho_{G,1}\mm$ have bounded support (because we assumed $\mu_0,\mu_1$ to have bounded support), yield
\[
-\int\rho_{r}^{-\frac1N}(\gamma_{r})\,\d\ppi_G(\gamma)\leq -\int\rho_0(\gamma_0)^{-\frac1N}\sigma^{(1-t)}_{K,N}(\sfd(\gamma_0,\gamma_1))+\rho_1(\gamma_1)^{-\frac1N}\sigma^{(t)}_{K,N}(\sfd(\gamma_0,\gamma_1))\,\d\ppi_G(\gamma),
\]
which, taking into account \eqref{eq:simpleresc}, is \eqref{eq:intcd}.

The assumption that $\mu_0,\mu_1$ have bounded support  can be removed with the same truncation argument used at the beginning of the proof. To deal with the case of arbitrary $0\leq t<s\leq 1$ use the uniqueness part of Theorem \ref{thm:optmap} again to deduce that the only optimal plan from $\mu_t$ to $\mu_s$ is given by $({\rm Restr}_t^s)_\sharp\ppi$, where ${\rm Restr}_t^s:\geo(X)\to\geo(X)$ is defined by
\[
{\rm Restr}_t^s(\gamma)_r:=\gamma_{(1-r)t+rs}.
\]
Then repeat the argument used for the case $t=0$, $s=1$.
\end{proof}

\begin{lemma}\label{le:ac}
Let $\mu,\nu\in\probt X$ be with bounded support and such that  $\mu\leq C\mm$ for some $C>0$. Then there exists a geodesic $(\mu_t)$ from $\mu$ to $\nu$ such that $\mu_t\ll\mm$ for every $t\in[0,1)$.
\end{lemma}
\begin{proof}
Let $(\nu^n)\subset\probt X$ be a sequence of absolutely continuous measures weakly converging to $\nu$ and with uniformly bounded supports and $\ppi^n\in\gopt(\mu,\nu^n)$ the unique optimal plan given by Theorem \ref{thm:optmap}. Then the bound \eqref{eq:boundpoint} gives, after dropping the term involving $\rho_1$, the inequality
\begin{equation}
\label{eq:sar}
\rho_t(\gamma_t)\leq \rho_0(\gamma_0)(\sigma_{K,N}^{(1-t)}(\sfd(\gamma_0,\gamma_1)))^{-N},\qquad\ppi^n-a.e.\ \gamma .
\end{equation}
By the definition of the distortion coefficients $\sigma_{K,N}^{(1-t)}(\theta)$ we see that for some function $f:[0,1)\to\R^+$ depending on $K$, $N$ and ${\rm diam}(\supp(\mu)\cup(\cup_n\supp(\nu^n)))$, it holds $(\sigma_{K,N}^{(1-t)}(\sfd(\gamma_0,\gamma_1)))^{-N}\leq f(t)$ and thus \eqref{eq:sar} and the bound $\mu\leq C\mm$ give
\[
\mu^n_t:=(\e_t)_\sharp\ppi^n\leq Cf(t)\mm,\qquad\forall t\in[0,1).
\]
This bound is independent on $n\in\N$, hence with a simple compactness argument based on the fact that $(X,\sfd,\mm)$ is proper we get the conclusion by letting $n\to\infty$.
\end{proof}

We shall also use the following lemma, whose proof was given in \cite{FigalliGigli11} (see also \cite{Gigli11}) for the case of Riemannian manifolds; yet, the argument is only metric and can be repeated without any change. We report it just for completeness.
\begin{lemma}\label{le:lip}
Let $(X,\sfd)$ be a proper geodesic space,   $\varphi$ a $c$-concave function and $\Omega\subset X$ the interior of $\{\varphi>-\infty\}$. Then $\varphi$ is locally bounded and locally Lipschitz on $\Omega$ and for every compact set $K\subset\Omega$ the set $\cup_{x\in K}\partial^c\varphi(x)$ is bounded and not empty.
\end{lemma}
\begin{proof} 
Being $c$-concave, $\varphi$ is the infimum of a family of continuous functions, hence upper-semicontinuous and thus locally bounded from above. We prove that it is locally bounded from below by contradiction. Thus, recall that $(X,\sfd)$ is proper, assume that there exists a sequence $(x_n)\subset \Omega$ converging to some $x_\infty\in\Omega$ such that $\varphi(x_n)\to-\infty$ as $n\to\infty$. For every $n\in\N$ let $y_n\in X$ be such that
\begin{equation}
\label{eq:yn}
\varphi(x_n)\geq \frac{\sfd^2(x_n,y_n)}{2}-\varphi^c(y_n)-1,\qquad\forall n\in\N,
\end{equation}
and notice that this bound and the fact that $\varphi(x_n)\to-\infty$ yield that $\varphi^c(y_n)\to+\infty$ as $n\to\infty$. Thus from 
\[
\R\ni \varphi(x_\infty)\leq \frac{\sfd^2(x_\infty,y_n)}{2}-\varphi^c(y_n),\qquad\forall n\in\N,
\]
we deduce that $\frac{\sfd^2(x_\infty,y_n)}{2}\to+\infty$ as well as $n\to\infty$ and therefore also that
\[
\lim_{n\to\infty}\frac{\sfd^2(x_n,y_n)}{2}\to+\infty.
\]
In particular, without loss of generality we can assume $\sfd(x_n,y_n)\geq 1$ for every $n\in\N$. Now let $\gamma^n:[0,\sfd(x_n,y_n)]\to X$ be a geodesic from $x_n$ to $y_n$ parametrized by arc-length.  We claim that 
\begin{equation}
\label{eq:claimcn}
\sup_{B_1(\gamma^n_1)}\varphi\to-\infty,\qquad \textrm{as } n\to\infty.
\end{equation}
Indeed, for $x\in B_1(\gamma^n_1) $  we have
\[
\begin{split}
\varphi(x)&\leq \frac{\sfd^2(x,y_n)}2-\varphi^c(y_n)\leq \frac{(\sfd(x,\gamma^n_1)+\sfd(\gamma^n_1,y_n))^2}2-\varphi^c(y_n)\\
&\leq \frac{\sfd^2(x_n,y_n)}2-\varphi^c(y_n)\leq\varphi(x_n)+1,
\end{split}
\]
having used \eqref{eq:yn} in the last inequality. Given that the $x_n$'s were chosen so that $\varphi(x_n)\to-\infty$ as $n\to\infty$, our claim \eqref{eq:claimcn} is proved.  

Up to pass to a subsequence, we can assume that $(\gamma^n_1)$ converges to some $z\in X$. From \eqref{eq:claimcn} it easily follows that in the internal part of $B_1(z)$ the function $\varphi$ is identically $-\infty$. Given that $\sfd(x,z)=1$, this fact contradicts the assumption that $x\in\Omega$. Hence $\varphi$ is locally bounded.

Now let $\bar x\in\Omega$ and $r>0$ be such that $B_{2r}(x)\subset\Omega$. Pick  $x\in B_r(\bar x)$ and let $(y_n)$ be such that $\varphi(x)=\lim_{n}\frac{\sfd^2(x,y_n)}{2}-\varphi^c(y_n)$. We claim that there exists a constant $C$ depending only on $\bar x, r$ and $\varphi$ such that $(y_n) \subset B_C(\bar x)$. In proving this we may assume that $\sfd(x,y_n)>r$ for all $n$. 
Pick unit speed geodesics $\gamma^n:[0,\sfd(x,y_n)]\to X$ from $x$ to $y_n$ and notice that
\[
\begin{split}
\lims_{n\to\infty}\varphi(x)-\varphi(\gamma^n_r)\geq \lims_{n\to\infty}\frac{\sfd^2(x,y_n)}{2}-\frac{\sfd^2(\gamma^n_r,y_n)}{2}=\lims_{n\to\infty}r\sfd(x,y_n)-\frac{r^2}{2}.
\end{split}
\]
By construction we have $x,\gamma^n_r\in B_{2r}(\bar x)\subset \Omega$ thus by what we previously proved we know that the leftmost side of the above inequality is bounded by some constant depending only on $\bar x, r$ and $\varphi$. Hence the sequence $(y_n)$ is bounded and we directly get that any limit point belongs to $\partial^c\varphi(x)$, which therefore is non-empty. The very same argument also shows that $C:=\cup_{x\in B_r(\bar x)}\partial^c\varphi(x)$ is bounded. In particular we get 
\[
\varphi(x)=\min_{y\in C}\frac{\sfd^2(x,y)}{2}-\varphi^c(y),\qquad\forall x\in B_r(\bar x),
\]
and since for $y\in C$ the functions $x\mapsto \frac{\sfd^2(x,y)}{2}-\varphi^c(y)$ are uniformly Lipschitz, we deduce the local Lipschitz continuity of $\varphi$ as well.
\end{proof}
\begin{theorem}[Exponentiation and optimal maps] Let $K\in \R$, $N\in[1,\infty)$, $(X,\sfd,\mm)$ an $\RCD^*(K,N)$ space,  $\varphi$ a $c$-concave function and $\Omega\subset X$ the interior of $\{\varphi>-\infty\}$. Then for $\mm$-a.e. $x\in\Omega$ there exists a unique geodesic $\gamma$ with $\gamma_0=x$ and $\gamma_1\in\partial^c\varphi(x)$.

In particular, for every $\mu,\nu\in\probt X$ with $\mu\ll\mm$ there exists a unique optimal geodesic plan $\ppi\in\gopt(\mu,\nu)$ and this plan is induced by a map and concentrated on a set of non-branching geodesics.
\end{theorem}
\begin{proof} Existence trivially follows from the fact that $\partial^c\varphi(x)$ is non-empty for every $x\in\Omega$ and the fact that $(X,\sfd)$ is geodesic. For uniqueness we argue by contradiction. For $x\in\Omega$ let  $G(x)\subset\geo(X)$ be the set of $\gamma$'s such that $\gamma_0=x$ and $\gamma_1\in\partial^c\varphi(x)$ and assume that there is a compact set $K_1\subset \Omega$ such that $\mm(K_1)>0$ and $\#G(x)\geq 2$ for every $x\in K_1$. By Lemma \ref{le:lip} we know that for some $L>0$ we have $\sfd(\gamma_0,\gamma_1)\leq L$ for any $x\in K_1$ and $\gamma\in G(x)$ so that the geodesics in $\cup_{x\in K_1}G(x)$ are equi-Lipschitz.

For some $a>0$ the compact set $K_2\subset K_1$ of $x$'s such that ${\rm diam}G(x)\geq a$ is such that $\mm(K_2)>0$. Pick such $a$ and $K_2$. For $t\in[0,1]$ put  $G_t(x):=\{\gamma_t:\gamma\in G(x)\}\subset X$ and consider the set $\mathcal K\subset K_2\times [0,1]$ of $(x,t)$'s such that  ${\rm diam}G_t(x)\geq\frac a2$. It is easy to check that $\mathcal K$ is closed and the continuity of geodesics grants that for any $x\in K_2$ the set of $t$'s such that $(x,t)\in \mathcal K$  has positive $\mathcal L^1$-measure. By Fubini's theorem, there is $t_0\in[0,1]$ such that the compact set $K_3\subset K_2$ of $x$'s such that ${\rm diam}G_{t_0}(x)\geq\frac a2$ has positive $\mm$-measure. Notice that necessarily $t_0>0$. With a Borel selection argument we can find a Borel map $T:K_3\to X$ such that $T(x)\in G_{t_0}(x)$ for every $x\in K_3$. Let $x_0\in X$ be such that  $T_\sharp(\mm\restr {K_3})(B_{\frac a6}(x_0))>0$ and put  $A:=T^{-1}(B_{\frac a6}(x_0))$, so that $\mm(A)>0$. By construction, the map $A\ni x\mapsto G_{t_0}(x)\setminus B_{\frac{a}3}(x_0)$ is Borel and has non-empty values, thus again with a Borel selection argument we can find Borel map $S:A\to X$ such that $S(x)\in G_{t_0}(x)\setminus B_{\frac{a}3}(x_0) $ for every $x\in A$.

Let $\mu:=\mm(A)^{-1}\mm\restr A$, $\nu_1:=T_\sharp\mu$ and $\nu_2:=S_\sharp\mu$. By construction $\nu_1$ and $\nu_2$ have disjoint support, and in particular $\nu_1\neq \nu_2$. Furthermore, the function $t_0\varphi$ is a Kantorovich potential both from $\mu$ to $\nu_1$ and from $\mu$ to $\nu_2$. Apply Lemma \ref{le:ac} to both $(\mu,\nu_1)$ and $(\mu,\nu_2)$ to find geodesics $(\mu^i_t)$, $i=1,2$, from $\mu$ to $\nu_1,\nu_2$ respectively such that $\mu^i_t\ll\mm$ for every $t\in[0,1)$, $i=1,2$. By construction, for $t$ sufficiently close to 1 we have $\mu^1_t\neq \mu^2_t$. Fix such $t$, let $\ppi^i\in\gopt(\mu,\mu^i_t)$, $i=1,2$ and notice that $\ppi^1\neq \ppi^2$ and that $\supp((\e_0,\e_1)_\sharp\ppi^i)\subset\partial^c(tt_0\varphi)$, $i=1,2$. 

Thus for the plan $\ppi:=\frac12(\ppi^1+\ppi^2)$ it also holds $\supp((\e_0,\e_1)_\sharp\ppi)\subset \partial^c(tt_0\varphi)$ and thus is optimal. Moreover it satisfies $(\e_0)_\sharp\ppi,(\e_1)_\sharp\ppi\ll\mm$ and, by construction, is not induced by a map. This contradicts Theorem \ref{thm:optmap}, concluding the proof of the first part of the statement.

For the second part, notice that if the optimal geodesic plan is not unique or not induced by a map, there must be $\ppi\in\gopt(\mu,\nu)$ which is not induced by a map. With a restriction argument we can then assume that $\mu:=(\e_0)_\sharp\ppi$, and $\nu:=(\e_1)_\sharp\ppi$ have bounded support, with $\mu\ll\mm$. But in this case there is a locally Lipschitz Kantorovich potential from $\mu$ to $\nu$ and the first part of the statement gives the conclusion.
\end{proof}

\end{document}